\documentclass[letter paper, 10 pt, conference]{ieeeconf}

\IEEEoverridecommandlockouts                              

\usepackage[utf8]{inputenc}
\usepackage{cite}
\usepackage{amsmath,amssymb,amsfonts}
\usepackage{steinmetz}
\usepackage{algorithmicx}
\usepackage{algorithm}
\usepackage{algpseudocode}
\usepackage{graphicx}
\usepackage{amsmath,amssymb}
\usepackage{textcomp}
\usepackage{algpseudocode} 
\usepackage{mathtools}
\usepackage[dvipsnames]{xcolor}
\usepackage{nicefrac}
\usepackage{txfonts}
\usepackage{todonotes}
\usepackage{url}

\title{\LARGE\bf Target Defense against Periodically Arriving Intruders}

\author{Arman Pourghorban and Dipankar Maity
\thanks{The authors are with the Department of Electrical and Computer Engineering at The University of North Carolina at Charlotte, NC, 28223, USA. 
Emails: 
{\tt apourgho@uncc.edu, dmaity@uncc.edu}
}%
\thanks{
This research is supported by the ARL grant ARL DCIST CRA W911NF-17-2-0181
}
}

\newcommand{\re}{\color{red}}

\newcommand{\rt}{\rho_{_{\rm T}}}
\newcommand{\ra}{\rho_{_{\rm A}}}
\newcommand{\xd}{\mathbf{x}_{\rm D}}
\newcommand{\xa}{\mathbf{x}_{\rm A}}
\newcommand{\xc}{\mathbf{x}_{\rm C}}
\newcommand{\thetae}{\theta_{\rm eng}}
\newcommand{\taue}{t_{\rm eng}}
\newcommand{\phie}{\phi_{\rm eng}}
\newcommand{\reng}{r_{\rm eng}}
\newcommand{\Se}{S_{\rm engage}}
\newcommand{\R}{\mathbb{R}^2}
\newcommand{\uvec}{\hat{\mathbf{u}}}
\newcommand{\x}{\textbf{x}}
\newcommand{\ro}{r_{_{\rm T}}}
\newcommand{\po}{p_{_{\Omega}}}
\newtheorem{lemma}{Lemma}

\newtheorem{definition}{Definition}

\newtheorem{remark}{Remark}
\begin{document}

\maketitle
\thispagestyle{empty}
\pagestyle{empty}
\begin{abstract}
We consider a variant of pursuit-evasion games where a single  defender is tasked to defend a static target from a sequence of periodically arriving intruders. The intruders' objective is to breach the boundary of a circular target without being captured and the defender's objective is to capture as many intruders as possible. 
At the beginning of each period, a new intruder appears at a random location on the perimeter of a fixed circle surrounding the target and moves radially towards the target center to breach the target.
The intruders are slower in speed compared to the defender and they have their own sensing footprint through which they can perfectly detect the defender if it is within their sensing range. 
 Considering the speed and sensing limitations of the agents, we analyze the entire game by dividing it into  \textit{partial information} and \textit{full information phases}.
 We address the defender's \textit{capturability} using the notions of \textit{engagement surface} and  \textit{capture circle}.
 We develop and analyze three efficient strategies for the defender and derive a lower bound on the capture fraction. 
Finally, we conduct a series of simulations and numerical experiments to compare and contrast the three proposed approaches.\\
\end{abstract}

\begin{keywords}
\small 
Target-defense, pursuit-evasion, partial information games, multi-agent systems.
\end{keywords}

\section{Introduction}
In this paper, we study a circular target-defence game between
a single defender and a team of sequentially arriving intruders.
This problem has been found useful in multi-robot applications such as
patrolling \cite{yan2016multi}, area monitoring \cite{robotics9020047}, area securing \cite{9632813}, coastline defense \cite{inproceedings} and has motivated a great amount of research e.g., see \cite{shishika2020review} for a review on this topic.
\par
Since the formulation of the Pursuit-Evasion Game (PEG)  \cite{isaacs1999differential}, several variations of it have been proposed in the last few decades that directly address target/area/perimeter defense type applications.
For example, reach-avoid games are a variant of PEGs where a group of agents attempt to reach a target while avoiding some
adversarial circumstances generated by the opponent group \cite{6426643,Fisac,garcia2020optimal}. 
 Likewise, perimeter-defence problems are another variant of PEGs wherein the defender team is tasked to capture the intruders before the latter breach the target perimeter. 
 Target-defence games, which are of particular relevance for this paper, have been studied by  
 considering the scenario where the defenders' movements are constrained on the perimeter of the target \cite{9620572,https://doi.org/10.48550/arxiv.2109.02852} and also where the defenders can move freely in the environment \cite{bajaj2021competitive},\cite{8279644}.
 In this work we consider the second scenario where the defenders can move freely in the environment. 
 
  Of particular relevance to this paper is the work in \cite{bajaj2019dynamic} that considered a problem where the intruders appeared  randomly according to a Poisson arrival process at a fixed distance from the target perimeter.
  Several algorithms for capturing the intruders are discussed and some performance bounds for the capture fraction are derived. 
  However, this work did not consider any sensing capability for the intruders and therefore, the intruders are restricted in terms of the evasive strategies they can deploy. 
  Similarly, in \cite{adler2022role} and \cite{macharet2020adaptive} among many others, the intruders attempting to breach a target perimeter are tasked with simple and fixed a priori strategies. 
  These works also do not consider the sensing capability which enables the intruders to evade from the defender instead of just being captured on their direct path toward the target. 
  To the best of our knowledge, some of the earliest works on similar problems have been considered in \cite{Bertsimas} where a vehicle (equivalent to our defender) is tasked to serve multiple randomly arriving customers (equivalent of our intruders) in an optimal fashion. 
  In contrast these works, we show in our earlier work \cite{arman} that the sensing capability and the adversarial nature of the arriving intruders
result in a particular type of evasive maneuver where the
intruders can force the defender to pursue them and thus, the
defender ends up capturing the intruders at locations which
are advantageous for the next intruders to increase their target
breaching probability.
Sensing-enabled motion tactics in the context of target defense games is one of the key features of this paper and our earlier works \cite{arman, 9561995}.
  
  In \cite{arman}, the arrival of the intruders were assumed to be sequential, i.e., the next intruder appears randomly only after the current intruder breaches the target perimeter or gets captured. 
  As a consequence, the defender had only one intruder to engage with at any given time. 
 In this paper, we extend the work of \cite{arman} and consider periodically incoming intruders. 
 This way, depending on the period of arrival, there might be multiple intruders available to engage with and the defender needs to decide which intruder to engage with while considering their sensing capabilities and the associated potential evasive strategies deployed by these intruders. 
  Similar to \cite{arman}, we consider a partial information game framework where, due to the sensing limitations, the agents do not have each other's information all the time and they must consider trade off between sensing advantage and positional advantage while choosing their strategies. 
The main contributions of this paper are: (i) We analyze a sensing limited perimeter-defence game against periodically incoming intruders.
The intruders' entry points are considered to be random.
We analyze both finite and infinite time (asymptotic) performances of the game. 
(ii) We derive the strategies for all the agents and the discuss how to determine the \textit{capturability} of an intruder using the concepts of \textit{engagement surface} and \textit{capture circle}. (iii) Based on the notion of \textit{capturability}, we introduce three strategies for the defender to determine which intruder(s) to purse and in what manner (i.e., how to delay being sensed by that intruder to gain positional and informational advantages), (iv) We theoretically derive a lower bound for one of our proposed strategies.
 (v) We numerically validate (using Monte-Carlo type random trials of experiments) the theoretically found lower bound and compare it with the true performances of the three proposed strategies for a wide range of parameters.
\par
The rest of the paper is organized as follows: In Section~\ref{sec:Preliminaries}, we formulate our problem, discuss parametric assumptions, and provide some useful definitions and necessary background materials. 
The two phases (\textit{partial} and \textit{full information}) are discussed in Sections~\ref{sec:partInfo} and \ref{sec:fullInfo}, respectively. 
We analyze the whole game and design algorithms for the defender's strategies in Section~\ref{sec:GameAnalysis}.
In section~\ref{sec:Lowebound}, we analyze one of the proposed defender strategies and compute a lower bound on this strategy's performance. 
 Simulations and numerical results are discussed  in Section~\ref{sec:Simu}. 
 The developed algorithms were deployed in Robotics Operating Systems (ROS) to generate experimental results.
We conclude the paper in Section~\ref{sec:Conclusion}. 
\par
\textit{Notation:} 
All vectors are denoted with lowercase bold symbols, e.g., $\mathbf{x}$. 
$\uvec(\theta)$ denotes the unit vector $[\cos\theta,~\sin\theta]^\intercal$.
\section{Preliminaries} \label{sec:Preliminaries}
\subsection{Problem Formulation}
\label{sec:ProbFormulation}
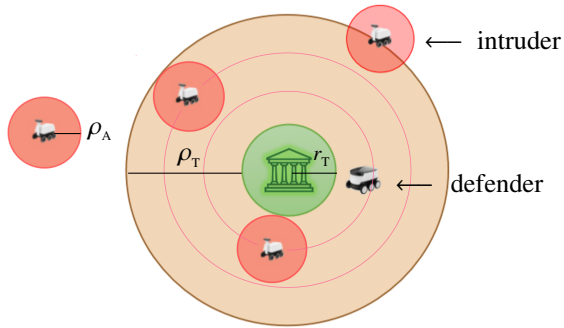
\begin{figure}
    \centering
        \includegraphics[ width = 0.075 \textwidth]{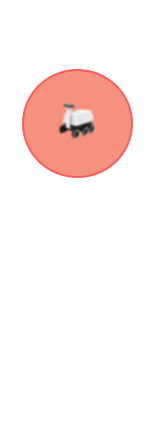}
    \includegraphics[ width = 0.28 \textwidth]{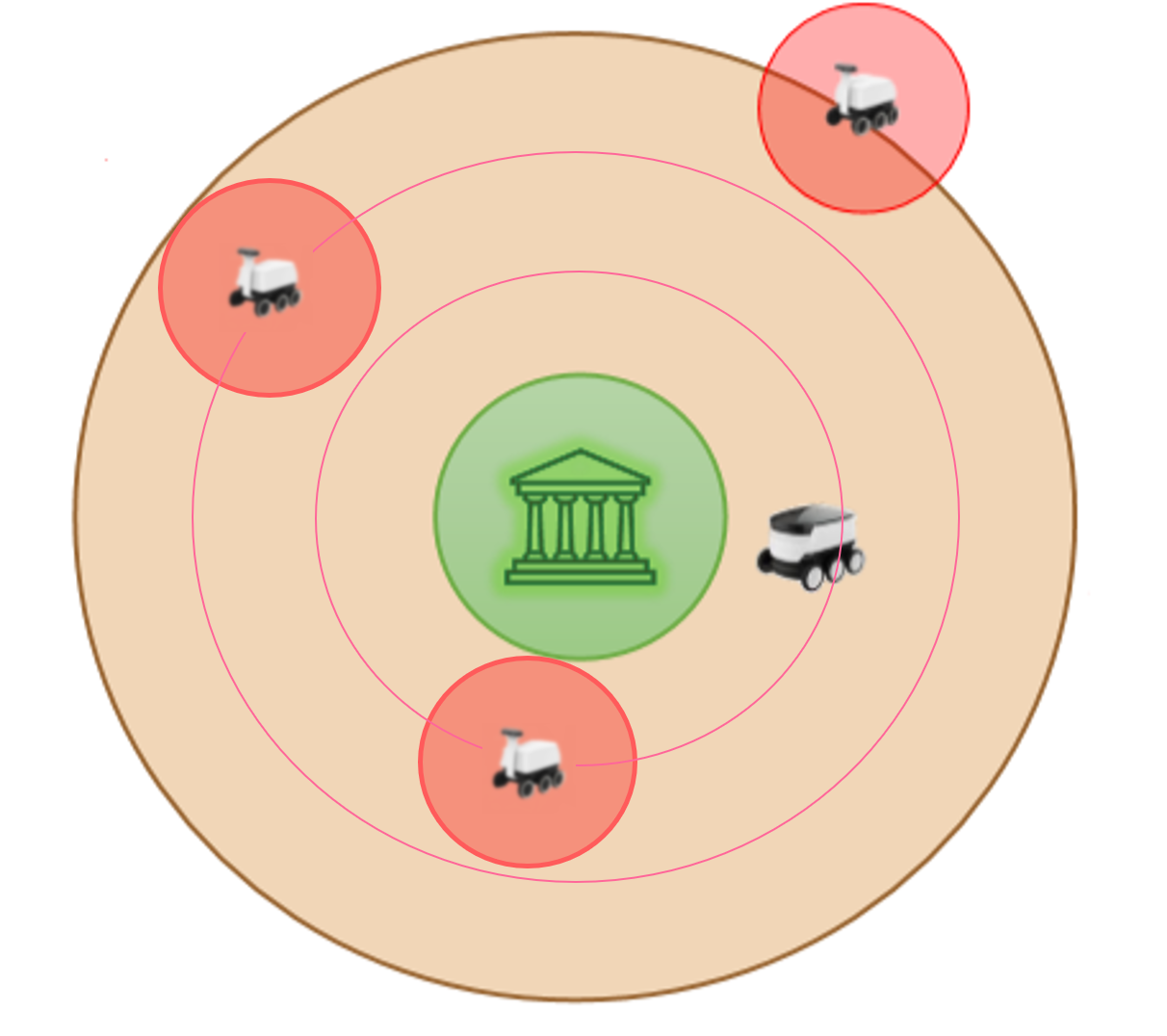}
    \put (-10, 55 )  {defender}
    \put (-31, 55 ) {$\longleftarrow$}
    \put (0, 109 )  {intruder}
    \put (-20, 109 ) {$\longleftarrow$}
       \put (-132, 62 ) {\line(1,0) {43}}
    \put (-113, 67) {$\rt$}
              \put (-160, 77 ) 
           {\textcolor{black}{{\line(1,0) {10}}}}
 \put (-147, 77) {$\ra$}
           \put (-70, 62 ) 
           {\textcolor{black}{{\line(1,0) {17}}}}
 \footnotesize{\put (-62, 67) {$\ro$}}

           \caption{\colorbox{green!30}{\textcolor{green!30}{g}}: Target region of radius $\ro$, \colorbox{red!20}{\textcolor{red!20}{g}}: intruders sensing region of radius $\ra$, \colorbox{orange!60!black!25}{\textcolor{orange!60!black!25}{g}}: Target Sensing Region (TSR) of radius $\rt$, \textcolor{orange!50!black!50}{\textbf{\Large --}} TSR Boundary.}
    \label{fig:introPic}
    \vspace{-12 pt}
\end{figure}
We consider a target guarding problem in ${\mathbb{R}^2}$ where a periodically arriving sequence of intruders is tasked to breach a circular target region $\mathcal{R}_{\rm T}=\{\x \in {\mathbb{R}^2}\ | \ \|\x\| \leq \ro\}$, see Fig.~\ref{fig:introPic}. 
A single defender is assigned to protect the target from the incoming sequence of intruders and, if possible, capture the intruders before they reach the target boundary $\partial \mathcal{R}_{\rm T}\triangleq\{\x \in {\mathbb{R}^2}\ | \ \|\x\| = \ro\}$. 
 Given an arrival period $T$, the defender maximizes the asymptotic intruder capture fraction
$
 J = \liminf_{t \to \infty} \frac{N(t)}{\lfloor \frac{t}{T} \rfloor},
$
 where $N(t)$ denotes the total number of intruders captured in $t$ amount of time.

There is a circular sensing annulus of radius $\rt$  around the target boundary such that  the defender can sense the location of any attacker present in this region.
We will refer to this annulus sensing region as the \textit{Target Sensing Region} (TSR). 
Each intruder appears on the TSR boundary with a uniform random probability which is {independent} of the former arrivals.  
We assume that the intruders do not communicate/coordinate with each other and they each move radially toward the target center.
\par Let $\xa(t),\xd(t) \in \R$ denote the positions of a representative intruder and the defender at time $t$. 
The defender and the intruders are assumed to have first-order dynamics, i.e., 
\begin{align}
    \dot{\mathbf{x}}_{\rm A} = v_{\rm A} \uvec(\psi_{\rm A}), \qquad \dot{\mathbf{x}}_{\rm D} = v_{\rm D} \uvec(\psi_{\rm D}),
\end{align}
where the defender (intruder) directly controls its speed and heading angle by selecting $v_{\rm D}$ and $\psi_{\rm D}$ ($v_{\rm A}$ and $\psi_{\rm A}$), respectively.
We assume that the defender and the intruders
 move with speeds 1 and $\nu$, respectively, i.e., $|v_{\rm D}(t)| = 1$ and $|v_{\rm A}(t)| = \nu$ for all $t$. 
Furthermore, to avoid a trivial scenario we assume that $\nu < 1 $, i.e., the defender is faster.

 After arriving on the TSR boundary, an intruder moves radially toward the target center until it senses the defender, similar to the setup of \cite{bajaj2019dynamic, macharet2020adaptive, 9561995} among others.
The defender has access to the instantaneous positions $\xa(t)$ for all the intruders that are within the TSR.
Each intruder is equipped with a sensor and is able to sense the defender only if the defender is within a distance of $\ra$ or less, i.e., the intruders have a circular sensing footprint of radius $\ra$.
Using this sensing capability, an intruder is able to find the right breaching point on the target, or able to get out of the TSR uncaptured, or is able to evade for some time before getting captured by the defender. 
 This evasive maneuver is an important capability for the intruders since it forces the defender to pursue and capture the intruder at a location that is likely to be unfavorable for the defender to start pursuing the next intruder.
  \subsection{Parametric Assumption}
   The overall outcome of
this game depends on the game
parameters $\ro,\ra $, $\rt $ and $\nu $. 
We assume that the parameters satisfy the following condition
\begin{align}
   (1+\nicefrac{2\nu}{(1-\nu^2)}) \ra\le \rt, \quad \nu \ro \le \rt.
    \label{eq:Asm}
\end{align}
This assumption is required for capturing a \textit{capturable} attacker inside the TSR, as will be discussed later in much detail in \ref{engsurf}. 
Capturing the attacker inside the TSR is a necessity for the defender since the defender has access to the attacker's location only when the latter is inside the TSR. 
Otherwise, there can be a deadlock situation where, if
the intruder cannot breach, it may successfully evade and keep
trying indefinitely until it successfully gets to breach or no longer has a strategy to evade.
Therefore, it is not possible to relax  assumption \eqref{eq:Asm} without preventing  \textit{evasion} of the attacker.
Furthermore, this assumption also avoids the situation where an incoming intruder can sense the defender as soon as it appears on the TSR boundary.
Such a scenario also leads to a deadlock because the intruder can immediately get out of the TSR if it does not have a strategy to breach. 
The readers are referred to \cite[Theorem~2]{arman} for more details on this assumption. 
The second condition (i.e., $\nu \ro\leq \rt$) is necessary to ensure that the defender can capture any intruder from the target center. 
This way, when the defender has no capturable intruders, it goes toward the target center to increase its probability to capture the next arriving intruder.
\subsection{Apollonius Circle}
Given the locations of the  defender and the intruder  at time $t$, the \textit{Apollonius Circle} contains all the points the intruder can reach before the defender gets there.
The set of all such points is a circular region  with center $\xc(t)$ and radius $r_{\rm C}(t)$:
\begin{align} \label{eq:AC}
    \xc(t) = \alpha\xa(t) - \beta \xd(t), \quad r_{\rm C}(t) = \gamma\|\xa(t) - \xd(t)\|,
\end{align}
where 
$
    \alpha = (1-\nu^2)^{-1}, \quad \gamma = \nu \alpha,\quad \beta = \nu\gamma.
$
\begin{lemma}[\!\!\cite{dorothy2021one}] \label{lem:arXiv}
The defender has a strategy to capture the intruder arbitrary close to the Apollonius circle, regardless of any strategy that intruder chooses to escape. \hfill $\triangle$
\end{lemma}

Apollonius circle is a critical tool for analyzing these types of games since it specifies the regions where the agents can have winning strategies.

\section{Game Phases}
Each game between the defender and any of the intruders can be divided into two phases, namely, the \textit{Full information phase} and the \textit{Partial/Asymmetric information phase}.
 From the time an intruder arrives on the TSR boundary to the time it first senses the defender is referred to as the \textit{partial information phase} between that intruder and the defender.
Therefore, in this phase, only the defender can sense the intruder  and not vice-versa.
%
Defender's strategy against that intruder in this phase depends on whether the intruder is \textit{capturable} or not. 
If it is capturable, then the defender picks a location and time to engage with the intruder by coming within the intruder's sensing region.
Otherwise, the defender picks another intruder or movers closer to the target center. 
\par 
On the other hand, in the \textit{Full information phase}, both agents can sense each other, i.e., $\|\xa(t)\|<\ro+\rt$ and $\|\xa(t)-\xd(t)\|\le \ra$. In this phase, the defender has a strategy to capture the intruder if the Apollonius circle fits inside the TSR.
However, if the Apollonius circle has any intersection with the target, the intruder would have a guaranteed strategy to breach, and if it has any intersection with the TSR boundary, then the intruder would have a guaranteed strategy to escape before being captured by the defender.

\subsection{Partial Information Phase} \label{sec:partInfo}
In this phase, the defender has access to the intruders'
information and not vice-versa.
The intruders move radially until it senses the defender.
To exploit this information advantage, the defender engages with the intruders in particular configurations that guarantee capture.
In \cite{arman} we show that earliest engagement is not the optimal strategy since the intruder may survive by evading. 
The defender may need to wait for a certain amount of time before it engages with the intruder. 
A sub-optimal waiting strategy was first introduced in \cite{9561995} and the optimal engagement strategy was derived in \cite{arman}. 
We will discuss these desirable engagement configurations and how to reach such configurations using the notion of \textit{Engagement Surface}.

\subsubsection{Engagement Surface} \label{engsurf}
Given an intruder location $\xa(t)$ in the partial information phase, we can accurately predict its path since it is moving in a straight line toward the target center with maximum velocity. 
Let $\xa(t_1)$ be the location of the intruder at time $t_1 > t$. 
To start the engagement with the intruder exactly at time $t_1$, the defender must be at point on the intruder's sensing boundary, i.e., $\xd(t_1) = \xa(t_1) + \ra\uvec({\theta})$ for some $\theta \in [0,2\pi)$.
The choice of $\theta$ (equivalently, the choice of $\xd(t_1)$) will determine the Apollonius circle, and the final outcome (i.e., capture/evade/breach) of the engagement.
We defined the notion of \textit{engagement surface} to represent this concept in \cite{arman}. 
\begin{definition}[Engagement Surface]
Given an intruder location $\xa(t_0)$ at time $t_0$, the \textit{engagement surface} for this intruder, denoted as $\Se(\xa(t_0))$, is the locus of the points $\xd \triangleq \xa(t_0+\taue)+\ra\uvec(\thetae)$, where $(\taue,\thetae)$ satisfies 
\begin{equation} \label{eq:engagementSurface}
    \sin^2\bigg(\frac{\thetae-\theta_{\rm A}}{2}\bigg) =   \frac{(\ro+\gamma\ra)^2 - (r_{\rm A}-\taue \nu -\beta\ra)^2}{4\beta\ra(r_{\rm A}-\taue \nu)},
\end{equation}
where $(r_{\rm A}, \theta_{\rm A})$ is the polar coordinate representation of the intruder's given location $\xa(t_0)$. \hfill $\triangle$
\end{definition}
%

Note that $\taue $ represents the time spent before the intruder senses the defender and the full information phase starts and $\thetae$ denotes the angle on the intruder’s sensing boundary the defender will engage. 
Since the pair $(\taue, \thetae)$ uniquely defines a spatio-temporal engagement point for a given intruder location $\xa(t_0)$, we will simply consider $(\taue, \thetae)$ satisfying \eqref{eq:engagementSurface} to be the elements of the set $\Se(\xa(t_0))$.
  A representative engagement surface is shown in Fig.~\ref{fig:engagement_surface}.

\begin{figure}
    \centering
   { \includegraphics[trim = 40 20 40 20, clip, width=0.27\textwidth]{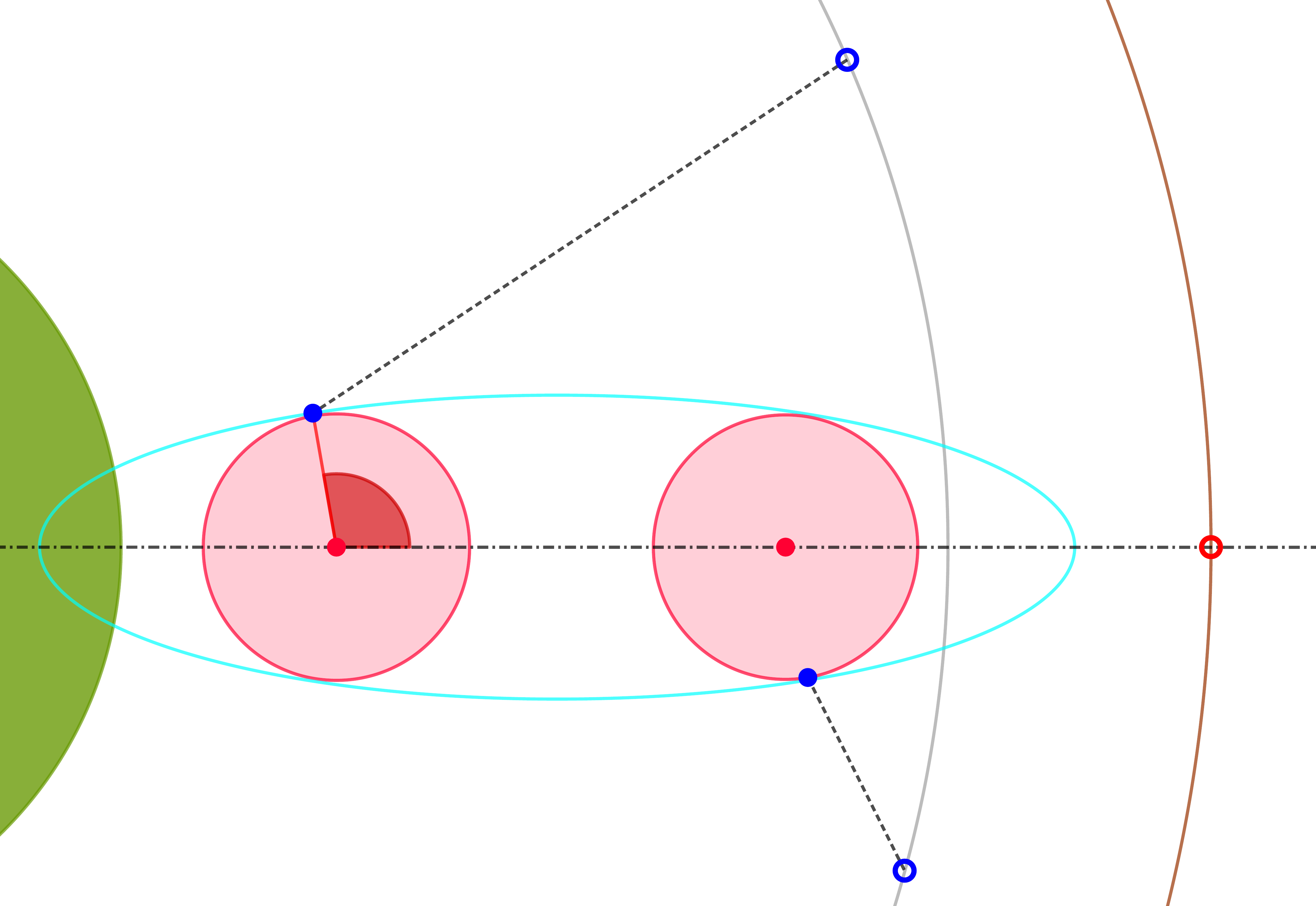}}
     \put(-108, 27) {\small \re $\xa(t_0\!+\!\taue)$}
     \put(-95, 45) {\small \re $\thetae$}
     \put(-110, 55) {\small  $x_1$}
     \put(-44, 75) {\small {$D_1$}}
      \put(-38, 0) {\small {$D_2$}}
      \put(-58, 15) {\small {$x_2$}}
    \hfill
    {\includegraphics[ width=0.19\textwidth]{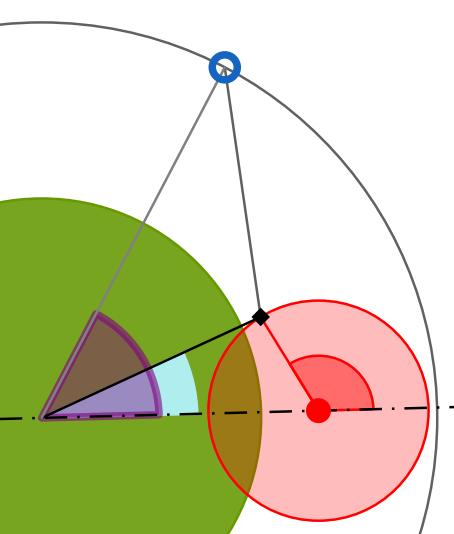}}
        \put(-70,70) {$r$}
    \put(-60,47) {$r_{\rm eng}$}
    \put(-40, 15) {$\xa(t_0\!+\!\taue^*)$}
    \put(-30, 40) {\re $\thetae^*$}
     \put(-65, 30) {\color{blue} $\phie$} 
     \put(-97, 40) {\color{purple!50!blue} $\theta_{\max}$} 
     \put(-40,52) {$\x_{\rm eng}$}
    \caption{ 
    (Left) The cyan oval-shaped curve represents the locus of the engagement points $\xa(t_0+\taue) + \ra\uvec(\thetae)$. 
    We demonstrate a configuration where a defender $D_1$  starts the full information phase by engaging at the point $x_1$ when the intruder is at $\xa(t_0+\taue)$.
        (Right) The engagement configuration with the intruder at time $t_0+\taue^*$ where {\re $\medbullet$ }  and $\Diamondblack$ denote the locations of the intruder and the defender, respectively.
    The hollow blue circle denotes the initial location of the defender.
    }
    \label{fig:engagement_surface}
    \vspace{- 10 pt}
\end{figure}

%
Starting the full information phase from any of the points on this engagement surface guarantees capture of the intruder
since the Apollonius circles generated from these engagement configurations do not have any intersection with the target perimeter or the TSR boundary.
This can be verified by constructing the Apollonius circle as per \eqref{eq:AC} and using Assumption \eqref{eq:Asm}.
Given that the capture is guaranteed if engagement is started from this surface, the intruder's objective in this case is to get captured the farthest from the target center.
This will increase the breach probability for the other intruders as discussed in \cite[Corollary 1]{arman}.
Furthermore, it has been shown in \cite{arman} that the farthest capture point will be at a radial location of $\ro + 2\gamma\ra$ since the radius of Apollonius circle will be $\gamma\ra$ and the center will at a distance of $\ro+\gamma\ra$ from the target center. 
The circle with radius $\ro+2\gamma\ra$ and concentric with the target is referred to as the \textit{Capture Circle}, since all the captures of the intruders will happen on this circle, see \cite{arman} for details.

\subsubsection{Capturability of an Intruder} \label{Max Angle}
Let the defender be located at a point $r\uvec({\theta_{\rm D}})$ at the time the intruder appears on the TSR boundary. In \cite[Theorem~1]{arman}, we proved the condition on $r$ and $\theta_{\rm D}$ that ensures capturing an incoming intruder utilizing the \textit{Engagement Surface}. This theorem provides the maximum allowed initial angular separation $\theta_{\max}$ between the intruder and the defender to
ensure capture by starting the engagement at $(\taue,\thetae) \in \Se(\xa(t_0))$:
\begin{subequations} \label{Maxanglesep}
\begin{align} 
    &\theta_{\max}(\taue,\thetae, r) = \cos^{-1} \bigg( \frac{ \reng^2 +r^2 - \taue^2}{2\reng r} \bigg)  + \phie, \\
    &\phie = \sin^{-1}\bigg(\frac{\ra \sin(\thetae)}{\reng} \bigg),  \\
    \begin{split}
    & \reng = \Big((\ro+\rt-\taue\nu)^2 +\ra^2 \\
    & \qquad\quad + 2(\ro+\rt-\taue\nu) \ra\cos(\thetae)\Big)^{1/2}. \qquad 
    \end{split}
    \end{align} 
     \end{subequations}
    We notice that $\theta_{\max}$ depends on the choice of $\taue$ and $\thetae$.
    By maximizing over $(\taue,\thetae)$ with the constraint \eqref{eq:engagementSurface}, we obtain the pair $(\taue^*, \thetae^*)$.
    \begin{remark}
    The importance of finding $\theta_{\max}(\taue^*,\thetae^*)$ is that, if the angular separation between an intruder and the defender is more than $\theta_{\max}(\taue^*,\thetae^*)$ at the time the intruder arrived on the TSR boundary, then it is ensured that the defender is not able to capture this intruder. 
    Therefore, the defender can simply ignore this intruder from the game and focus on the rest. 
    If the angular separation is exactly $\theta_{\max}(\taue^*,\thetae^*)$, then $(\taue^*,\thetae^*)$ is the unique point where the defender needs to start the engagement. 
    Otherwise, if the angular separation is less, the defender has multiple engagement points and it will be able to start engagement sooner than $\taue^*$.
    Thus, $\taue^*$ is a tight upper bound on the amount of time spent before the engagement starts. \hfill $\triangle$
    \end{remark}
    
    Equation \eqref{Maxanglesep} is computed for an intruder that just arrived on the TSR boundary. 
    We now generalize this concept to find the maximum angular separation to be able capture an intruder located at an arbitrary radial location of $r_{\rm A}$. 
    That is, given a defender located at $r\uvec(\theta_{\rm D})$, it is able to capture an intruder currently located at $r_{\rm A}\uvec(\theta_{\rm A})$ by starting the full information phase $\taue$ amount of time from now, if $|\theta_{\rm D} - \theta_{\rm A}|\le \theta_{\max}(\taue,\thetae, r, r_{\rm A})$, where $\theta_{\max}(\taue,\thetae, r, r_{\rm A})$ has the same expression as in \eqref{Maxanglesep} with $r_{\rm eng}$ being
    \begin{align*}
        \reng=\big((r_{\rm A}-\taue\nu)^2+\ra^2+2(r_{\rm A}-\taue\nu)\ra \cos(\thetae)\big)^{1/2}.
    \end{align*}
    By maximizing over $(\taue,\thetae)$ with the constraint \eqref{eq:engagementSurface}, we find $\theta_{\max}(\taue^o,\thetae^o, r, r_{\rm A})$ that provides the \textit{capturability condition} for an intruder located at a radial distance of $r_{\rm A}$. 
    Using $\theta_{\max}(\taue^o,\thetae^o, r, r_{\rm A})$, we define the \textit{capturable set} $\Omega_{\rm cap}$ for a given defender location $r\uvec(\theta_{\rm D})$ as follows
    \begin{align}
        \Omega_{\rm cap}(r,\theta_{\rm D}) =\{ (r_{\rm A}, \theta_{\rm A}) ~:|\theta_{\rm A}\!-\!\theta_{\rm D}|\le  \theta_{\max}\!(\taue^o,\thetae^o, r, r_{\rm A}\!)\}.
    \end{align}
        If an intruder is located at $r_{\rm A}\uvec(\theta_{\rm A})$ such that $(r_{\rm A}, \theta_{\rm A}) \in \Omega_{\rm cap}(r,\theta_{\rm D})$, then that intruder is capturable.
        Thus, by computing this region $\Omega_{\rm cap}(r,\theta_{\rm D})$, the defender can easily check which intruders are capturable and which are not. 
    In Fig.~\ref{fig:capturability}, we illustrate this region. 
     \begin{figure}
    \centering
     {\includegraphics[trim = 10 10 10 10, clip, height = 3.4 cm, angle = 270, origin = c]{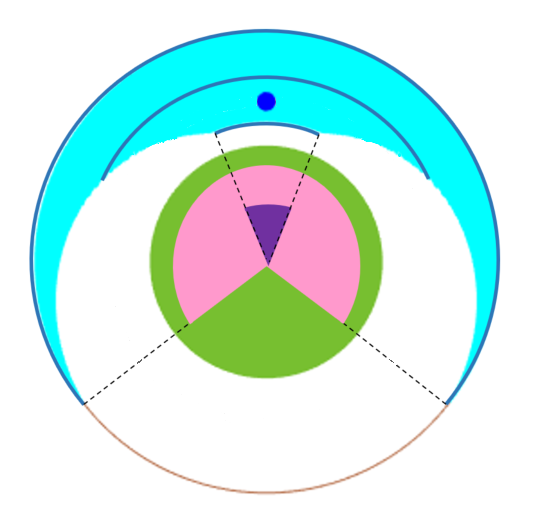}}
     \put (-32,62) {$o_1$}
     \put (-43,80) {$o_2$}
     \put (-88,85) {$o_k$}
     \vspace{ - 6 pt}
    \caption{The cyan region represents $\Omega_{\rm cap}$ for a defender where {\color{blue} $\medbullet$} denotes its location.
    The blue curves represent the intersection of each orbit $o_i$ with the region $\Omega_{\rm cap}$ which form the capturable arcs, as discussed in Section~\ref{sec:Lowebound}.
    The purple and the pink angular sectors represent the angles ($\varphi_1$ and $\varphi_k$, respectively) formed by the first and the last orbits.} \label{fig:capturability}
\vspace{-10 pt}
\end{figure}
    
    \begin{remark} \label{rem:OmegaD}
    $\Omega_{\rm cap}(r,\theta_{\rm D})$ is symmetric about $\theta_{\rm D}$ and furthermore, due to the circular symmetry, $\Omega_{\rm cap}(r,\theta_{\rm D})$ obtained by rotating $\Omega_{\rm cap}(r,0)$ by angle $\theta_{\rm D}$, for every $r$. \hfill $\triangle$
    \end{remark}
    
    \begin{remark}
    Let the defender be located at a point $r\uvec({\theta_{\rm D}})$.
    For any intruder located at $r_{\rm A}\uvec(\theta_{\rm A})$, the defender is able capture this intruder  if $(r_{\rm A}, \theta_{\rm A}) \in \Omega_{\rm cap}(r,\theta_{\rm D})$.
    Furthermore, the full information phase will start no later than $\taue^o$.
    Note that $\taue^o$ is only a function of the radial locations of the intruder ($r_{\rm A}$) and the defender ($r$).
    Furthermore, for all $r_{\rm A}< \rt+ \ro$, it can be verified that $\taue^o < \nicefrac{\ra}{(1+\nu)} + \nicefrac{(r_{\rm A} - \ro)}{\nu}$. 
    \hfill $\triangle$
    \end{remark}

\subsection{Full Information Phase} \label{sec:fullInfo}
In this phase of the game, both agents can sense each other. As soon as the intruder senses the defender it computes the instantaneous Apollonius circle and checks whether this configuration leads to a capture or breach by checking whether this circle intersects with the target or the TSR boundary, respectively.

Let the intruder and defender locations at the start of the engagement phase be $\xa$ and $\xd = \xa + \ra\uvec(\thetae)$, respectively.
Therefore, the Apollonius circle center will be at $\xc = \alpha\xa - \beta \xd = \xa -\beta \ra \uvec(\thetae)$.
Let us denote $\xa \equiv r_{\rm A}\uvec(\theta_{\rm A})$ and $\xc \equiv r_{\rm C} \uvec(\theta_{\rm C})$, and consequently,
\begin{align} \label{eq:thetaC}
\theta_{\rm C} = \tan^{-1}\bigg( \frac{r_{\rm A}\sin(\theta_{\rm A}) - \beta\ra \sin({\thetae})}{r_{\rm A}\cos(\theta_{\rm A}) - \beta\ra \cos({\thetae})}\bigg).
\end{align}
Therefore, the point $\x_p$ on the Apollonius circle that is farthest from the target center will be at 
\begin{align} \label{eq:captureLocation}
   \x_p = \xc + \gamma \ra \uvec( \theta_{\rm C}), 
\end{align}
since the radius of the Apollonius circle is $\gamma \ra$.

Since the engagement point is chosen to be on the engagement surface, one may verify that $\|\x_p\| = \ro + 2\gamma\ra$.
That is every capture will occur on a circle of radius $\ro + 2\gamma\ra$.
This helps us analyzing the game since, after capturing an intruder, the defender always starts it next pursuit from this circle. 
Therefore, the defender can precompute the capturability region $\Omega_{\rm cap}(\ro +2\gamma\ra, 0)$ and using Remark~\ref{rem:OmegaD} can compute its instantaneous capturable region for any $\theta_{\rm D}$.

\begin{remark}
Although $\x_p$ is the point that is furthest from the target, the intruder still may chose to not move towards $\x_p$ and get captured at a different point on the Apollonius circle. 
Based on Lemma~\ref{lem:arXiv}, we can construct a pursuit strategy for the defender which will ensure that the intruder cannot get out of the Apollonius circle.
For the subsequent analysis, we will assume that the intruder gets captured at $\x_p$. \hfill $\triangle$
\end{remark}
\par
 During this phase, an intruder (say intruder $i$) might sense the defender while the defender is pursuing another intruder. 
 At this moment, intruder $i$ constructs the Apollonius circle and checks whether it can breach.
 If it cannot breach but escape is possible, then this intruder confirms that the defender has targeted another intruder since a rational defender will not waste time by engaging with an intruder that can escape.
 If neither breach nor evasion is possible, the intruder moves toward the point $\x_p$ for an infinitesimal amount of time and checks whether the defender has also moved toward $\x_p$.
 If the defender did not move toward $\x_p$, the intruder concludes that the defender is pursuing another intruder and hence intruder $i$ resumes moving radially toward the target.
 Otherwise, they continue their motion toward $\x_p$ and capture happens.



\section{Defender Strategy} \label{sec:GameAnalysis}
The game begins with the defender being located at the center of the target.
The first intruder appears on the TSR boundary at  time $t=0$. 
Since this intruder is guaranteed to be captured due to Assumption \eqref{eq:Asm}, the defender moves to the closest engagement point and waits until the intruder senses the defender. As soon as the intruder senses the defender, the \textit{full information phase} starts and capture would take place on the \textit{Capture Circle} as discussed in Section~\ref{sec:fullInfo}.
One may also verify that this intruder will be captured at time $t=(\nicefrac{\rt}{\nu})+2\gamma \ra$. 
New intruders will appear randomly on the TSR boundary at every $T$ amount of time, where $T$ is the arrival period.
By the time the first intruder is captured, $\lfloor \nicefrac{(\nicefrac{\rt}{\nu}+2\gamma \ra)}{T} \rfloor$ new intruders would have appeared and, depending on the game parameters, some of them may even have breached while the defender was pursuing the first one, e.g., if $2\gamma\ra <T$, then the second intruder would have breached.
After the first capture, the defender is located on the capture circle and makes a decision on how and which intruder to pursue. 
In this paper we propose three strategies for the defender to decide which intruder to pursue, namely: \textit{Nearest Agent}, \textit{Earliest Breach}, and \textit{Weighted Distance} strategy.


\subsection{Nearest Agent Strategy}
In this strategy, the defender considers the intruder nearest to it to pursue.
If the nearest intruder was capturable -- which can be easily decided by checking whether that intruder lies within $\Omega_{\rm cap}$ -- the defender will find the earliest engagement point $(\taue,\thetae)$ that lies on the engagement surface of this intruder and reachable by this defender, i.e., given $\xa(t_0)$, the defender solves the following optimization problem
\begin{align} \label{eq:optimization}
\begin{split}
    \min\nolimits_{\taue,~\thetae} ~~~& \taue\\
    \text{subject to  }~~& (\taue, \thetae) \in \Se(\xa(t_0)), \\
    &\|\xa(t_0+\taue)+\ra\uvec(\thetae) - \xd(t_0)\| \le \taue.
\end{split}
\end{align}
The defender start the full information phase at this earliest possible engagement location. 
In this way, the defender minimizes the time to start the engagement phase so that it has more time pursue the other intruders. 
If the nearest intruder is not capturable, it checks whether the second nearest one is capturable, and so on. If none of the available intruders in the TSR is capturable, i.e., none of the intruders are within $\Omega_{\rm cap}$, 
the defender moves radially toward the center of the target until a new intruder appears on the TSR boundary and then the defender checks the capturability of this newly arrived intruder.
If this intruder is not capturable then the defender keeps moving toward the target center and will check capturability of the next intruders appearing on the TSR boundary. 
This process repeats until an intruder arrives that is capturable, or the defender is less than $(\nicefrac{\rt}{\nu})-\ro$ distance away from the target center, in which case it is guaranteed to capture any incoming intruder \cite[Lemma~2]{arman}.
For any capturable intruder, the defender picks the earliest reachable engagement point, since minimizing the capture time gives the defender more time to pursue other intruders.

\begin{remark}
The duration of the game between the defender and an intruder is comprised of two components: the time ($\taue$) spent in the partial information phase and the time to capture the intruder once the full information phase starts. 
When capture occurs at $\x_p$ defined in \eqref{eq:captureLocation}, the capture time only depends on the choice of $\thetae$, and given by the expression $t_{\rm cap} \triangleq \alpha\ra\!\sqrt{1+\nu^2 + 2\nu\cos(\thetae\! -\! \theta_{\rm C})}$, where $\theta_{\rm C}$ is defined in \eqref{eq:thetaC}.
Therefore, one could consider $\taue + t_{\rm cap}$ in \eqref{eq:optimization} to minimize the total game duration instead of minimizing the duration of the partial information phase ($\taue$). 
\hfill $\triangle$
\end{remark}

The optimization problem in \eqref{eq:optimization} is solved numerically via exhaustive search in this paper.
Efficient techniques can be developed that leverages the geometry of the engagement surface and left as a potential future work.

\subsection{Earliest Breach Strategy}
This strategy considers the intruders' distances from the center of the target. 
The objective is to prioritize the intruder that is the closest to breach, which is analogous to the concept of \textit{Earliest Deadline First} for dynamic priority scheduling algorithms in real-time systems.
In this approach, the defender first computes the region $\Omega_{\rm cap}$ and considers the intruder in this region that is closest to the target center.
Similar to the \textit{Nearest Agent Strategy}, the defender's objective here is also to minimize the capture time, therefore, it picks the earliest engagement point.
If $\Omega_{\rm cap}$ does not contain any intruder, the defender moves toward the target center until the next intruder appears on the TSR boundary. 

Notice that the major difference between the \textit{Nearest Neighbor} and \textit{Earliest Breach} strategies is how the defender prioritizes the intruders within $\Omega_{\rm cap}$. 
If $\Omega_{\rm cap}$ is does not contain any intruder, then both the algorithms will behave exactly the same.


       


\begin{algorithm} [t]
    \caption{Initialization $\&$ Intruder Arrival}
    \label{euclid}
    \begin{algorithmic}[1] 
    \State Initialize $\xd \gets [0,0]^\intercal$, $N_{\rm capture}\gets 0$, $n\gets 1$, $\tilde\rt \gets (\ro + \rt)$, and $N , T$ \Comment{$N$ = total no. of arrivals}
    \For{$t = 0:T: NT$}
    \State $\theta_{\rm A}(n) \sim {\mathcal U}(-\pi, \pi)$  \Comment{Uniform random arrival}
    \State $\xa(n) \gets \tilde\rt\uvec(\theta_{\rm A}(n))$
    \State $n \gets n+1$
    \EndFor
    \end{algorithmic}
    \end{algorithm}
\begin{algorithm}[t]
    \caption{Defender’s strategy} \label{algorithm3}
    \begin{algorithmic}[1] 
    \State Initialize the weight $w \in [0,1]$ and the arrival process (Algorithm~\ref{euclid})
    \State $t_0  \gets $ current time,
    \State {\tt dist} $\gets \emptyset$  \Comment List of capturable intruders
    \For{$ n=1 : m$ } \Comment{ $m$ = No. of intruders within TSR}
    \If{$\xa{_{,\,n}}(t_0) \in \Omega_{\rm cap} (\xd(t_0))$ } \Comment {{\color{blue!60}\textit{capturability} test}}
      \State {\tt dist}$_n \gets w \|\xa{_{,\,n}}(t_0)\| + (1-w) \|\xa{_{,\,n}}(t_0) - \xd(t_0)\|$
    \EndIf
    \EndFor
    \If {{\tt dist}$\ne \emptyset$ }
    \State $\displaystyle i^* \gets \arg\!\!\min_{n=1:m} {\tt dist}_n$ \Comment Pick the `closest' intruder
    \State $(\taue, \thetae) \gets$ Solution from \eqref{eq:optimization} for $\xa{_{,\,i^*}}$
    \State Defender goes to ${\x}_{\rm eng}$ \Comment \textit{Full Info.} phase starts
        \State $N_{\rm capture}\gets N_{\rm capture} +1$ \Comment{{\color{blue!60}capture happens}}
        \State $\xd \gets \x_p$ \Comment{$\xd$ after capture}
    \Else
    \State Defender gets closer to the target center until the next intruder appears on the TSR boundary
    \EndIf
    
    
    \State Repeat from Step 2
    \end{algorithmic}
\end{algorithm}
\subsection{Weighted Distance Strategy}
In this strategy, we propose a hybrid between the \textit{Nearest Agent} and the \textit{Earliest Breach} strategies. In this approach,  the defender instead of checking the intruders' distances from the target or from itself, it considers a convex combination of these distances and picks the intruder within $\Omega_{\rm cap}$ that has the lowest weighed distance. 
This weighted distance approach is a generalized method that contains the previous two strategies as special cases by appropriately choosing the weight. 
By adaptively changing the weights, the defender will be able to efficiently adjust its behavior for more complicated scenarios such as where the arrival is aperiodic.
The pseudocode is presented in Algorithms~\ref{euclid} and~\ref{algorithm3}. 
\begin{remark}
Note that all the algorithms have the same computational complexity since they only differ in the weight $w$ used in line~6 of Algorithm~\ref{algorithm3}.
\end{remark}
%

    

       

\section{Performance Guarantee on Earliest Breach strategy} \label{sec:Lowebound}

We measure the performance of the proposed strategies by
their capture fraction. Capture fraction is the ratio of the number of intruders
captured and the number of intruders arrived. In this section,
we compute a lower bound on this capture fraction for the \textit{Earliest Breach} strategy.

Recall from our previous discussion that the defender minimizes the engagement time, which depends on the location of the intruder.
Therefore, the actual engagement time is a random variable due to the random arrival points of the intruder. 
The probability distribution of the earliest engagement time is analytically and computationally intractable. 
 To analytically compute a lower bound on the capture fraction, we use an upper bound on this random variable. 
  We further assume that, after each capture, the defender waits for a new intruder to appear on the TSR boundary before it checks whether $\Omega_{\rm cap}$ contains any intruder or not.
  By this way, we are ensured that the intruders will be located on a group of \textit{fixed} circular orbits $o_1,o_2,...,o_k$ with $T\nu$ distance from each other and the first orbit being the TSR boundary itself.
  As shown in Fig.~\ref{fig:capturability}, for a given defender location, just an arc on each orbit is capturable. 
  These arcs are the intersections between the orbits and the set $\Omega_{\rm cap}$. 
  
  
  Let $\varphi_i$ be the angle made by the arc associated with orbit $o_i$; see Fig.~\ref{fig:capturability}. 
  Therefore, 
  \[
  \varphi_i = \theta_{\max}(\taue^o, \thetae^o, r, r_{\rm A})\big|_{r = \ro+2\gamma\ra,~r_{\rm A}= \ro+\rt-(i-1)T\nu}.
  \]
  The probability an intruder being located on that arc of orbit $o_i$ is $\nicefrac{\varphi_i}{2\pi}$, since the intruders appear uniformly randomly on the TSR boundary and move toward the target center radially afterwards.
    Let $\po$ denote the probability that there is at least one intruder on one of these arcs. 
    Therefore, 
    \[
    p_{\Omega} = 1- \prod\nolimits_{i=1}^{k}(1-q_i),
    \]
  where $k = \lceil \nicefrac{\rt}{(T\nu)} \rceil$ is the total number of orbits.
  
  With $1- \po$ probability, there will not be any intruder within $\Omega_{\rm cap}$ and therefore, the defender will move toward the target center and check capturability of the next arriving intruder after $T$ amount of time. 
  The defender keeps moving toward the target center and keeps checking capturability as soon as new intruder arrives. 
  Let $p_i$ denote the probability that the defender has been unsuccessful to find a capturable intruder $i$ times in a row. 
  Therefore, the defender will be at a radial location of $\ro+2\gamma\ra - i T$.
  Note that, this probability only depends on the intruder arriving on the TSR boundary, and therefore is computed based on the maximum angular separation formula \eqref{Maxanglesep}, i.e., $p_i=\nicefrac{\theta_{\max}(\taue^*,\thetae^*,~ \ro+2\gamma\ra - i T)}{2\pi}$.
  We now present the lower bound of the \textit{Earliest Breach} strategy in the following lemma.

 \begin{lemma} \label{lem:lowerBound}
  For a given arrival period $T$, a lower bound on the capture fraction of the \textit{Earliest Breach} strategy is
 \begin{align} \label{centerlowbound}
          c_\infty \triangleq \frac{T (1-\nu)}{\ra+(1-\nu)\tau_{\rm avg}},
 \end{align}
 where $\tau_{\rm avg}$ is derived in Lemma~\ref{lem:tauAvg}. \hfill $\triangle$
 \end{lemma}

 \begin{proof}
 Let us consider that the defender just captured its $n$-th intruder and has consumed $T_n$ amount of time, which is a random variable.
 The total time consumed after capturing the $n+1$-th intruder is therefore
 \begin{align*}
     T_{n+1}=T_n+\tau_{{\rm eng},\,\, (n+1)}+\tau_{{\rm cap},\,\, (n+1)},
 \end{align*}
 where $\tau_{{\rm eng},\,\, (n+1)}$ is the time spent before the $n+1$-th full information phase starts and $\tau_{{\rm {{cap}}},\,\, (n+1)}$ is the time to capture the intruder as soon as both agents engage (i.e., the duration of the \textit{full information phase}).
 The latter is upper bounded by $\nicefrac{\ra}{1-\nu}$.
 Therefore we may write
 \begin{align*}
       T_{n+1} &\le T_n+\tau_{{\rm eng},\,\, (n+1)}+\frac{\ra}{1-\nu} \le T_1+n\frac{\ra}{1-\nu}+\sum\nolimits_{i=2}^{n+1} \tau_{{\rm eng},\,i}. 
 \end{align*}
 We notice that the time consumed for capturing the first intruder is a deterministic quantity
 $T_1 = \frac{\rt}{\nu}-\frac{\ra}{1+\nu}$.
 
 The number of arrived intruders by time $T_{n+1}$ is $\lfloor\frac{T_{n+1}}{T}\rfloor$.
 Therefore, the expected capture fraction after capturing the $(n+1)$-th one can be written as
\begin{align}
c_{n+1}= \mathbb{E}\Biggl[\frac{n+1}{\lfloor\frac{T_{n+1}}{T}\rfloor}\Biggr],
\end{align}
which can be further simplified to 
\begin{align*}
    c_{n+1} \geq \mathbb{E}\Biggl[\frac{n+1}{\frac{T_{1}}{T}+\frac{n}{T}\frac{\ra}{1-\nu}+\frac{1}{T}\sum\nolimits_{i=2}^{n+1} \tau_{{\rm eng},\,i} - \frac{\epsilon}{T}}\Biggr],
\end{align*}
where $\epsilon \in [0,T)$ is a random variable such that $\nicefrac{(T_{n+1}-\epsilon)}{T}$ is an integer. 
Using Jensen's inequality, we may write 
\begin{align}
    c_{n+1} &\geq \Biggl[\frac{n+1}{\frac{T_{1}}{T}+\frac{n}{T}\frac{\ra}{1-\nu}+\frac{1}{T}\sum\nolimits_{i=2}^{n+1} \mathbb{E}[\tau_{{\rm eng},\,i}] - \frac{\mathbb{E}[\epsilon]}{T}}\Biggr] \nonumber \\
    &\geq \Biggl[\frac{n+1}{\frac{T_{1}}{T}+\frac{n}{T}\frac{\ra}{1-\nu}+\frac{n}{T}\tau_{\rm avg} - \frac{\mathbb{E}[\epsilon]}{T}}\Biggr], \label{eq:cn}
\end{align}
where $\tau_{\rm avg}$ is an upper bound to $\mathbb{E}[\tau_{{\rm eng},\,i}]$, which is given in Lemma~\ref{lem:tauAvg}.
Taking the limit $n \to \infty$, we obtain
\begin{align*}
     \lim_{n \to \infty} c_{n+1} \ge \frac{T (1-\nu)}{\ra+(1-\nu)\tau_{\rm avg}} = c_\infty.
\end{align*}
 \end{proof}
\begin{remark}
Lemma~\ref{lem:lowerBound} provides a lower bound for both finite $n$ as well as when $n\to \infty$. 
As seen from \eqref{eq:cn}, the effects of $T_1$ and $\epsilon$ become negligible as $n$ increases. 
This behavior will also be noticed in the simulation (see Fig.~\ref{compare}) where the capture fraction has a transient behavior that stabilizes as the game progresses.
\hfill $\triangle$
\end{remark}

The following lemma provides the upper bound $\tau_{\rm avg}$.
 \begin{lemma} \label{lem:tauAvg}
 $\mathbb{E}\big[\tau_{{\rm eng},\,i} \big]$ is upper bounded by $\tau_{\rm avg}$, where 
 \begin{align*}
  \tau_{\rm avg}=&\po \taue^*+(1-\po) \sum\nolimits_{j=1}^\ell \prod\nolimits_{i=1}^{j-1} (1-p_i)p_{j} (\taue^*+jT) 
 \end{align*}
{where $\ell = \Big\lceil\frac{2\ro+2\gamma\ra-\frac{\rt}{\nu}}{T}\Big\rceil$ and $\prod_{i=1}^0 \equiv 1$.} \hfill $\triangle$
 \end{lemma}
 
 \begin{proof}
 A detailed proof of this lemma has been omitted due to page limitations.
 A sketch proof is provided below.
 If there is an intruder within $\Omega_{\rm cap}$ then the maximum time to start the engagement phase is $\taue^*$. 
 Otherwise, the defender move toward the target center and try again when a new intruder arrives $T$ amount of time later and so on until the defender reaches a radial location of $(\nicefrac{\rt}{\nu})-\ro$ or less, from where it is guaranteed to capture any intruder appearing on the TSR boundary. 
 \end{proof}

\section{Simulation Results} \label{sec:Simu}
We simulate the game with the following parameters $\ro =4$, $\ra = 1,~ \rt = 5$, and $\nu = 0.75 $. 
Under this parametric choice we conducted $100$ random trials of the game.
In each trial of  {{this Monte-Carlo experiment,}} we considered a game duration of $10^4$ units with a period of $T=2$. 
The abscissa of Fig.~\ref{compare} denotes the game duration and the ordinate denotes the capture fraction up to that time.
We have compared our three proposed strategies where recall that $w=0$ corresponds to the \textit{Earliest Breach strategy} and $w=1$ for the \textit{Nearest Agent strategy}.
We have considered $w=0.25,0.5$, and $0.75$ and simulated the \textit{Weighted Distance strategy}. 
For a short game duration  ($<10^2$), it is observed form Fig.~\ref{compare} that the defender performs the best under a \textit{Weighted Distance} strategy with weight $w=0.25$ and for a longer duration, we observe that the defender performs better with the \textit{Earliest Breach strategy}, although the difference in performance is small. 
Although one might be tempted to conclude that a lower value of $w$ is preferable, we also notice that $w=1$ performs reasonably well for both short and long horizons and outperforms $w=0$ for shorter horizons.
\begin{figure}
    \centering
    \includegraphics[width = 0.45\textwidth]{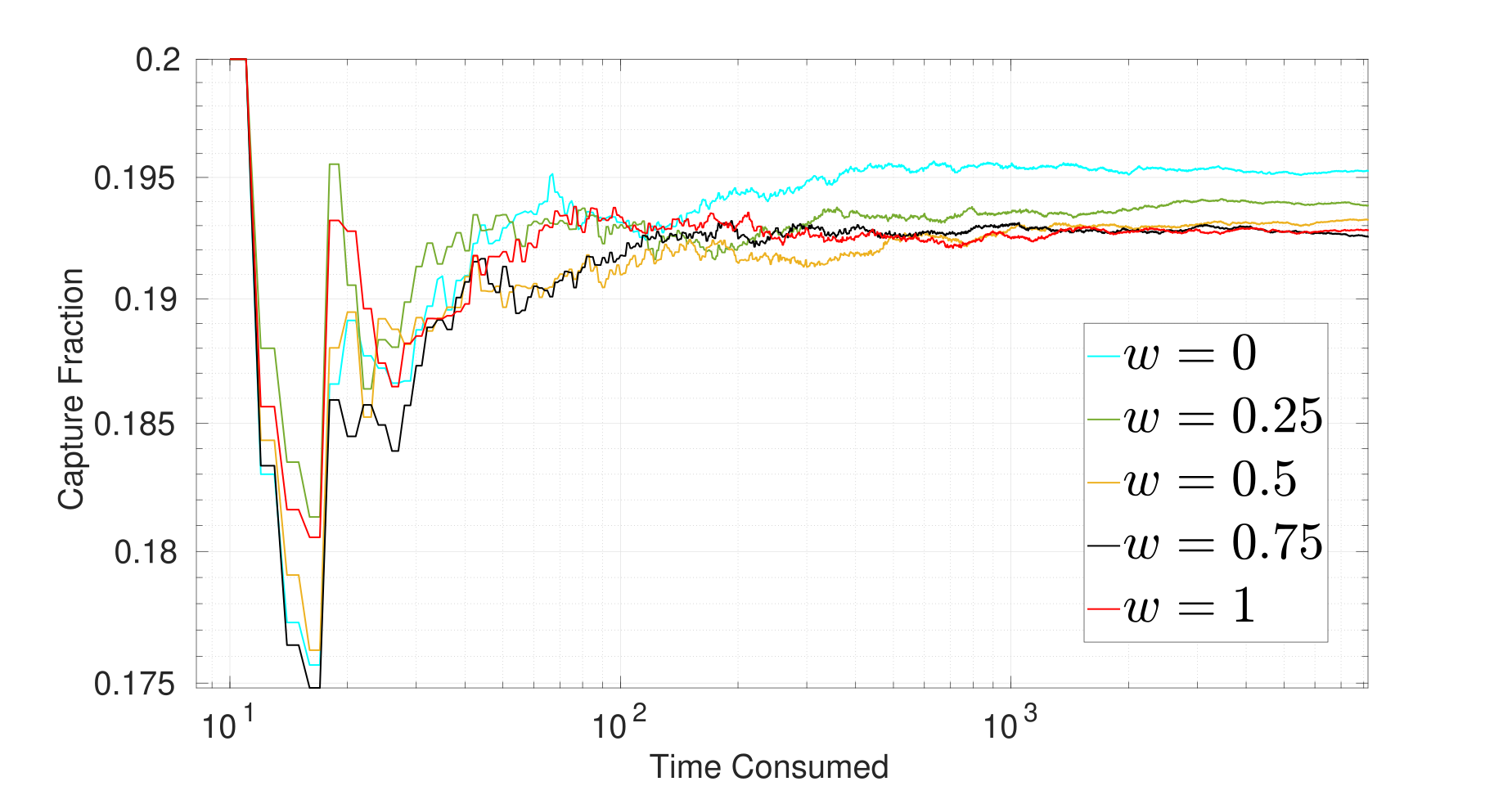}
    \caption{Capture fraction versus Game duration for $10^4$ unit time of the game. Cyan curve presented by $w=0$ is the \textit{Earliest Breach strategy} and the red curve presented by $w=1$ is the \textit{ Nearest Agent strategy}.\textit{Weighted Distance strategy} with three coefficients $w=0.25,0.5,0.75$ are shown with green, yellow and black curves respectively.}
    \label{compare}
    \vspace{-12 pt}
\end{figure}

Keeping all the other parameters the same, we now vary the period $T$ and compare the performances of the proposed algorithms, as shown in Fig.~\ref{periodcompare}.
In addition, the lower bound on on the \textit{Earliest Breach} strategy, as derived in Lemma~\ref{centerlowbound}, is also plotted in this figure. 
For small (integer) periods starting from 1 to 4, the defender captured more utilizing the \textit{Earliest Breach strategy}. 
For $T=5$, the \textit{Weighted Distance strategy} with $w=0.75$ has the highest capture fraction and for $T=6,7$ and $8$ the defender captured more using the \textit{Nearest Agent strategy}. 
From the period $T=9$ the capture fraction utilizing any of the algorithms converged to a same number. 
For $T \ge 13$ the defender is able to capture all the arriving intruders and hence the capture fraction became 1.     

\subsection{Robotic Operating System Experiments} We conducted a simulation in the Robotics Operating System (ROS) for $10$ incoming intruders for a game duration of $20$ unit.
The arrival period $T$ was chosen to be $2$.
To compare our two strategies (\textit{Earliest Breach} and \textit{Nearest Agent}), we considered the exact scenario including the arrival points of the intruders.
The defender performed differently for the two strategies.
The first intruder is captured the first by both the strategies. 
Afterwards,  the \textit{Earliest Breach strategy} picked the fifth intruder, which was captured next. 
Next, that strategy picked the ninth intruder but the capture was not completed within the game duration or 20 time unit. 
For the \textit{Nearest Agent strategy}, defender was able to capture both the sixth and eighth arriving intruders.
Then the algorithm picked the tenth one, however, before the defender could make noticeable progress, the game duration ended.
The trajectories of the defender for the \textit{Nearest Agent strategy} and the \textit{Earliest Breach strategy} are demonstrated in Fig.~\ref{gazebo}.\footnote{ 
A short simulation video is available at {\scriptsize \url{https://drive.google.com/drive/folders/1Gz6L53BLB6qwfLYzvQDr5LeQEoXwsEya?usp=sharing}}. }
 \begin{figure}
     \centering
\includegraphics[width = 0.5\textwidth]{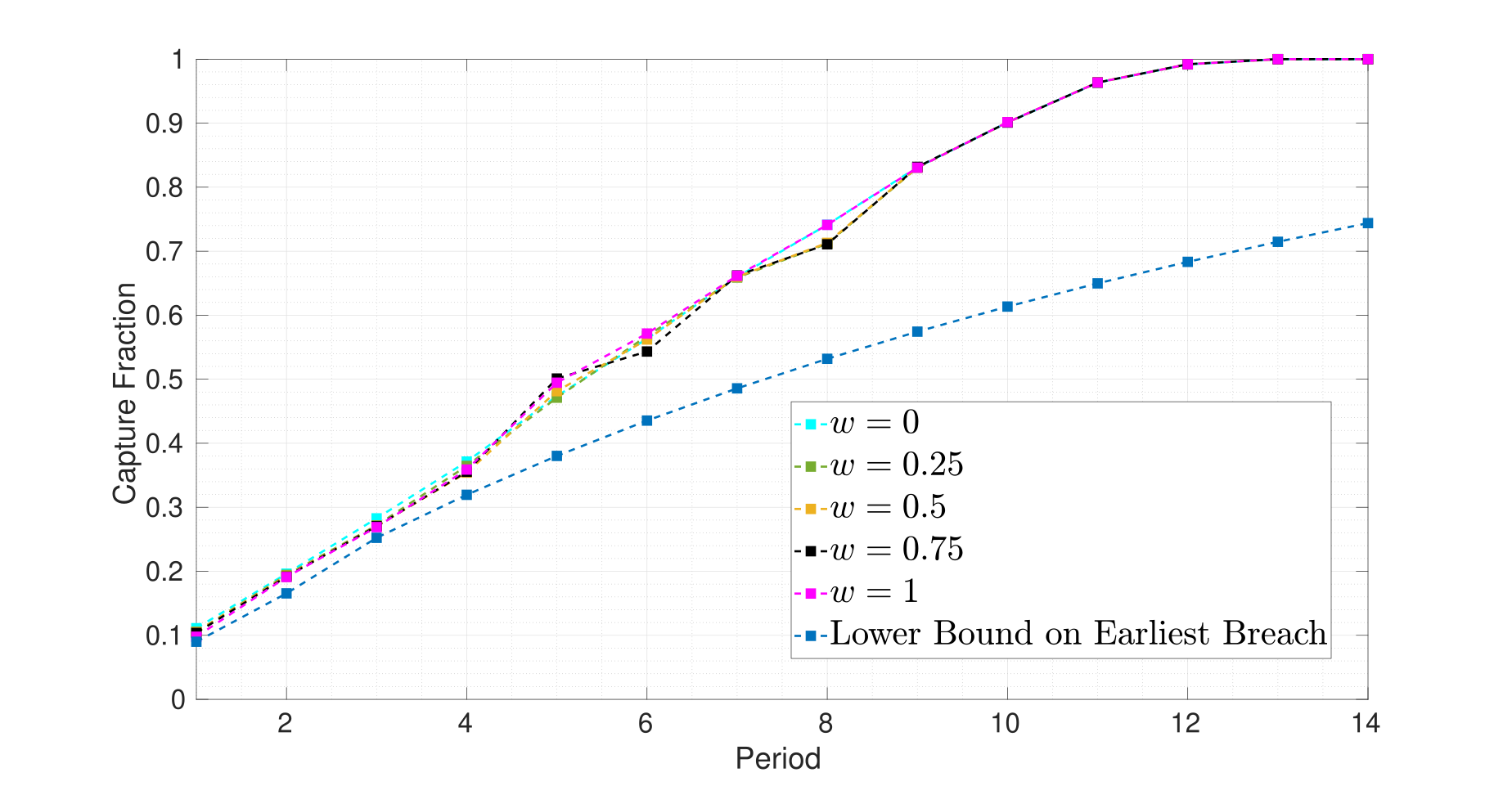}
         \caption{The capture fraction is plotted against the intruder arrival period $T$. 
         The cyan (dashed) line represents the case $w=0$, i.e., the \textit{Earliest Breach strategy} and the magenta one represents $w=1$, i.e., the \textit{Nearest Agent strategy}.
         The \textit{Weighted Distance strategy} with three coefficients $w=0.25,0.5$ and $0.75$ are shown with green, yellow, and black dashed lines respectively. 
         The derived lower bound on the capture fraction is shown in the blue dashed line.}
    \label{periodcompare}
 \end{figure}
 
\begin{figure}
    \centering
    \includegraphics[width = 0.48\textwidth]{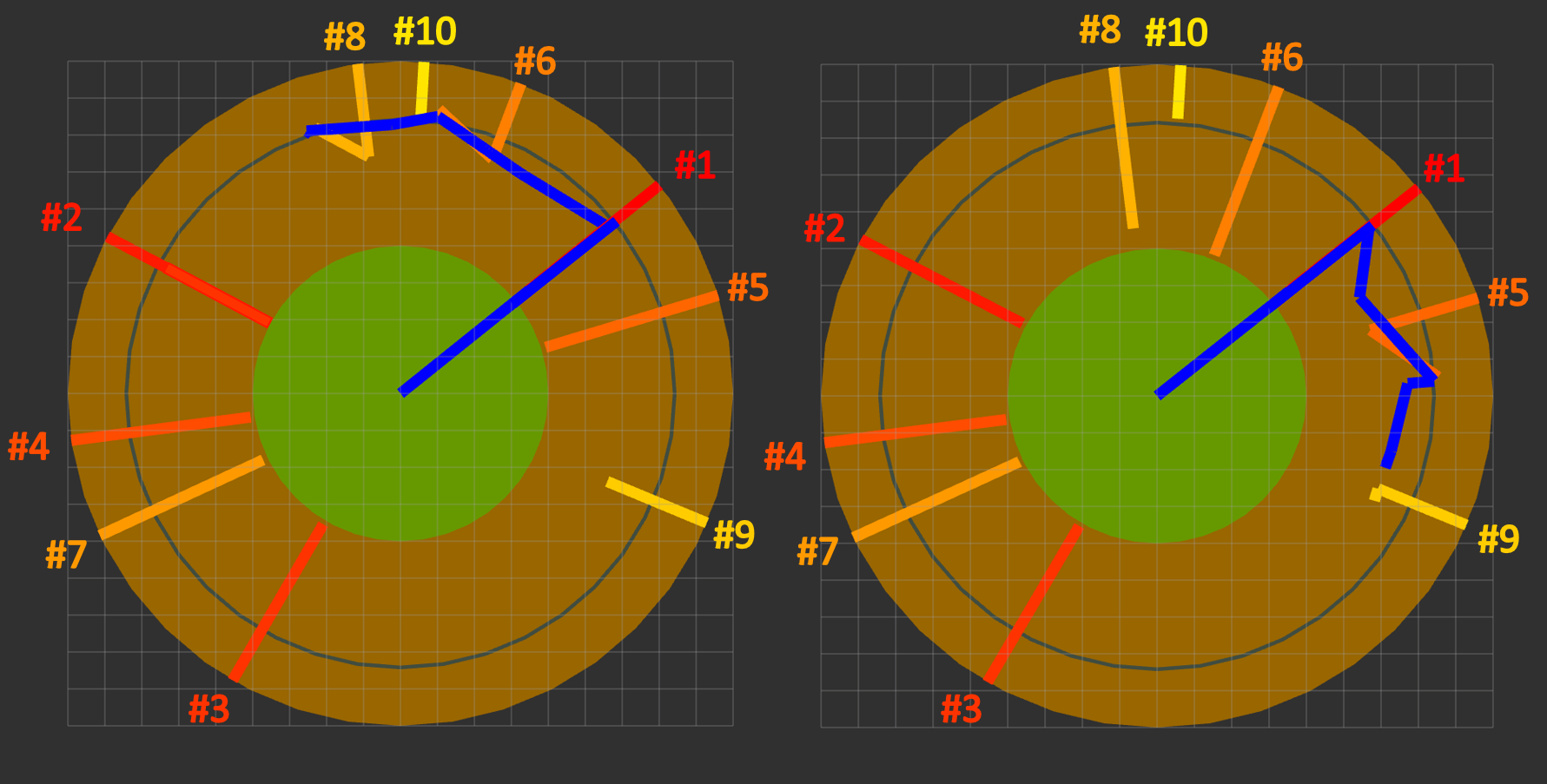}
    \caption{
    (Left) \textit{Nearest {{Agent}}}, (Right) \textit{Earliest Breach}.
    The sequence of incoming intruders is enumerated with colored numbers from 1 to 10.
    The trajectory of the defender is shown with blue lines. 
    The gray arc is the \textit{capture circle} on which all the intruders get captured. 
    }
    \vspace{- 10 pt}
\label{gazebo}
\end{figure}
\section{Conclusions} \label{sec:Conclusion}
In this paper, we formulated a target defense game against a sequence of periodically incoming intruders. 
Intruders move radially toward the target center to breach the target boundary while the defender is tasked to capture as many intruders as possible.
Based on the available information to the defender, it computes the \textit{capturability} of an intruder using the notion of \textit{Engagement Surface}, and all of the captures occur on the perimeter of a fixed circle called the \textit{Capture Circle}. 
We proposed a generalized distance-based strategy for the defender to prioritize the intruders to pursue and analytically computed a lower bound on the capture fraction. Numerical studies are presented to compare and contrast the performance over different arrival periods. 

A natural extension of this work would be to consider different arrival patterns for the intruders (e.g. non-uniform probability of arrival locations, multiple simultaneous arrivals). Furthermore, one may also consider a heterogeneous team for the intruders where different intruders may have different speed and sensing capabilities. 
In addition to these, extensions of this work to arbitrary shaped target would be of extremely useful for real-world applications. 

\bibliographystyle{IEEEtran}
\bibliography{reference}

\end{document}